\documentclass{article}
\usepackage[all,pdf]{xy}
\usepackage{latexsym} 
\usepackage{bm} 
\usepackage{amsmath}
\usepackage{amssymb}
\usepackage{amsthm}
\usepackage{dsfont}
\usepackage{indentfirst}
\newtheorem{theorem}{Theorem}[section]
 \newtheorem{lemma}[theorem]{Lemma}
 \newtheorem{proposition}[theorem]{Proposition}

 \newtheorem{remark}{Remark}[section]
 \newtheorem{example}{Example}[section]

\usepackage[body={13cm,21cm}]{geometry}
\usepackage[initials]{amsrefs}

\usepackage{amssymb,amscd}
\usepackage[all]{xy}
\usepackage{nicefrac}
\usepackage{units}
\usepackage{tikz-cd}

\usepackage[mathscr]{eucal}
\usepackage{enumerate}

\numberwithin{equation}{section}
\theoremstyle{plain}

\theoremstyle{definition}
\newtheorem{Def}[equation]{Definition}
\theoremstyle{remark}

\DeclareMathOperator{\Aut}{Aut}
\DeclareMathOperator{\Br}{Br}

\DeclareMathOperator{\Gal}{Gal}
\DeclareMathOperator{\GL}{GL}

\def\Br{{\rm Br\,}}

\def\Q{{\mathbb Q}}

\def\Gal{{\rm Gal}}

\def\ker {{\rm  Ker}}

\DeclareFontFamily{U}{wncy}{}
\DeclareFontShape{U}{wncy}{m}{n}{%
<5>wncyr5%
<6>wncyr6%
<7>wncyr7%
<8>wncyr8%
<9>wncyr9%
<10>wncyr10%
<11>wncyr10%
<12>wncyr6%
<14>wncyr7%
<17>wncyr8%
<20>wncyr10%
<25>wncyr10}{}
\DeclareMathAlphabet{\cyr}{U}{wncy}{m}{n}

\begin{document}

\title{Counting lattice points in central simple algebras with a given characteristic polynomial}

\author{Jiaqi Xie}

\date{}

\maketitle

\begin{abstract}
    We extend the asymptotic formula for counting integral matrices with a given irreducible characteristic polynomial by Eskin, Mozes and Shah  in \cite{ref1}  to the case of counting elements in a maximal order of certain central simple algebra with a given irreducible characteristic polynomial. 
\end{abstract}

\section{Introduction}

Eskin,  Mozes and Shah in \cite{ref1} counted the integral points in certain homogeneous spaces and determined the precise asymptotic formula for integral matrices with a given characteristic polynomial. More precisely, let $p(\lambda)$ be a monic  irreducible polynomial of degree $n\geq 2$ over $\mathbb{Z}$ and define 
 $$N(p(\lambda), T)=\#\{(x_{ij}) \in M_n(\mathbb{Z}):  \ \det(\lambda I_n -(x_{ij}))= p(\lambda) , \ \sqrt{\sum\limits_{1\leq i, j\leq n}x_{ij}^{2}} \leq T\}$$
 for $T>0$.  If all roots of $p(\lambda)$ are reals and the ring $\mathbb{Z}[\lambda]/(p(\lambda))$ is integrally closed, then 
 $$N(p(\lambda), T)\thicksim \frac{2^{n-1}h_KR_K\omega_n}{|d_{K}|^{\frac{1}{2}}\prod\limits_{i=2}^{n}\Lambda (\frac{i}{2})} \cdot T^{\frac{n(n-1)}{2}}$$
as $T\rightarrow \infty$, where $h_K$, $R_K$ and $d_{K}$ are the class number, the regulator and the discriminant of number field $K=\mathbb{Q}[X]/(p(X))$ respectively, $\omega_n$ is the volume of the unit ball in $\mathbb{R}^{\frac{n(n-1)}{2}}$ and $\Lambda (s)=\pi^{-s}\varGamma (s) \zeta(2s)$. 
In \cite[Theorem 1.2]{ref9}, Shah further removed the assumption that all roots of $p(\lambda)$ are reals and obtained the similar result by replacing $2^{n-1}$ with $2^{r_1}(2\pi)^{r_2}/w$ where $r_1$ is the number of real places of $K$, $r_2$ is the number of complex places of $K$ and $w$ is the number of roots of unity in $K$.  

In spirit of the local-global principle,  Wei and Xu in \cite{ref11} studied this counting problem by applying strong approximation with Brauer-Manin obstruction and pointed out that the above asymptotic formula can be obtained by the product of numbers of all local solutions. 
In this paper, we will extend the above asymptotic formula to central simple algebras based on  \cite{ref11}.  

Let $\mathcal{A}$ be a central simple algebra of degree $n$ over $\mathbb{Q}$. By \cite[(9.3) Theorem]{Rein}, one can define the (reduced) characteristic polynomial $f_a(\lambda)$ for each $a\in \mathcal A$ which is a monic polynomial of degree $n$ over $\mathbb Q$.  Moreover, if $a$ belongs to an order of $\mathcal A$, then $f_a(\lambda)$ is a monic polynomial of degree $n$ over $\mathbb Z$ by \cite[(10.1) Theorem]{Rein}.

Fix an maximal order $\mathfrak o_{\mathcal A}$ of $\mathcal A$ and assume $\mathcal A\otimes_{\mathbb Q} \mathbb R \cong  M_n(\mathbb R)$. Fix an imbedding 
$$\iota: \mathcal A \rightarrow \mathcal A\otimes_{\mathbb Q} \mathbb R \cong  M_n(\mathbb R) ; \ \ a \mapsto a\otimes 1 \in M_n(\mathbb R) . $$  For any given monic polynomial $p(\lambda)$ of degree $n$ over $\mathbb Z$, we define
$$ N_{\mathfrak o_{\mathcal A} }(p(\lambda), T)=\#\{a\in \mathfrak o_{\mathcal A}:  f_a(\lambda)= p(\lambda), \iota(a)=(x_{ij}) \in M_n(\mathbb{R}),  \sqrt{\sum\limits_{1\leq i, j\leq n}x_{ij}^{2}} \leq T\}$$
for $T>0$ and $\mathcal{A}_{p}=\mathcal{A}\otimes \mathbb{Q}_{p}$. The following is the main result of this paper.

\begin{theorem}
    When $\mathcal{A}_{p}$ is a matrix algebra or a division algebra for all primes $p$ and the ring $\mathbb{Z}[\lambda]/(p(\lambda))$ is integrally closed, then
    $$N_{\mathfrak o_{\mathcal A} }(p(\lambda), T)\thicksim (\prod\limits_{p\in S}\frac{n}{e_{p}\cdot \prod\limits_{i=1}^{n-1}(1-\frac{1}{p^{i}})})\cdot\frac{2^{r_{1}}(2\pi)^{r_{2}}h_KR_K\omega_n}{w|d_{K}|^{\frac{1}{2}}\prod\limits_{i=2}^{n}\Lambda (\frac{i}{2})} \cdot T^{\frac{n(n-1)}{2}}$$
    as $T\rightarrow \infty$, where $S=\{p:\mathcal{A}_{p}\ is\ a\ division\ algebra\ over\ \mathbb{Q}_{p}\}$ and $e_p$ is ramification index of extension $\Bbb Q_p[\lambda]/(p(\lambda))$ over $\Bbb Q_p$.
\end{theorem}

Notation and Terminology are standard. For a central simple algebra $\mathcal{A}$ of degree $n$ over $\mathbb{Q}$, we use $G_{\mathcal A}$ to denote a linear algebraic over $\mathbb Q$ which represents the functor
$$ (( \mathbb Q - \text{algebras})) \rightarrow ((\text{groups})); \ \ E \mapsto (E\otimes_{\mathbb Q} A)^\times  $$
and $SG_{\mathcal A}$ to be a closed subgroup of $G_{\mathcal A}$ which represents the functor
$$ (( \mathbb Q - \text{algebras})) \rightarrow ((\text{groups})); \ \ E \mapsto \ker ((E\otimes_{\mathbb Q} A)^\times\xrightarrow{N} E^\times)   $$
where $N$ is the reduced norm map. One has the exact sequence of linear algebraic groups
 \begin{equation}\label{exact}  1\longrightarrow SG_{\mathcal A} \longrightarrow G_{\mathcal A} \longrightarrow \mathbb G_m \longrightarrow 1 .  \end{equation}

For a central simple algebra $\mathcal{A}$ of degree $n$ over $\mathbb{Q}$ and a monic irreducible polynomial $p(\lambda)$ of degree $n$ over $\mathbb Z$, we define a variety $X_{\mathcal A, p(\lambda)}$ over $\mathbb Q$ which represents the functor 
$$ (( \mathbb Q - \text{algebras})) \rightarrow ((\text{sets})); \ \ E \mapsto \{ x\in E\otimes_{\mathbb Q} \mathcal A: \ f_x(\lambda)=p(\lambda) \} $$ 
where $f_x(\lambda)$ is the characteristic polynomial of $x$

For a fixed maximal order $\mathfrak o_{\mathcal A}$ of $\mathcal A$, we define an integral model $\bf X$ of $X_{\mathcal A, p(\lambda)}$ over $\mathbb Z$ which represents the functor 
$$ (( \mathbb Z - \text{algebras})) \rightarrow ((\text{sets})); \ \ R \mapsto \{ x\in R\otimes_{\mathbb Z} \mathfrak o_{\mathcal A}: \ f_x(\lambda)=p(\lambda) \} . $$

Both $G_{\mathcal A}$ and $SG_{\mathcal A}$ act on $X_{\mathcal A, p(\lambda)}$ over $\mathbb Q$ by conjugation and $X_{\mathcal A, p(\lambda)}$ is a homogeneous space of $G_{\mathcal A}$ and $SG_{\mathcal A}$.  Moreover, we have
$$ {\bf X}(\mathbb Z)= \{ x\in \mathfrak o_{\mathcal A}: \ f_x(\lambda)= p(\lambda) \} . $$

We also use $\mathbb A_{\mathbb Q}$ and $\mathbb I_{\mathbb Q}$ to denote the adeles and ideles of $\mathbb Q$.

\section{Quantitive local-global principle}

In this section, we provide the quantitive local-global principle for counting integral points in $\bf X$ defined in Section 1. 

\begin{Def}\label{integral}  Suppose that $\mathcal{A}$ is  a central simple algebra of degree $n$ over $\mathbb{Q}$ satisfying $\mathcal A\otimes_{\mathbb Q} \mathbb R \cong M_n(\mathbb R)$ and $p(\lambda)$ is a monic irreducible polynomial of degree $n$ over $\mathbb Z$.  Let $X_{\mathcal A, p(\lambda)}$ be a variety  over $\mathbb Q$ and $\bf X$ be an integral model of $X_{\mathcal A, p(\lambda)}$ defined in Section 1.  Write 
$$N(\textbf{X},T)=\sharp \{ a\in\textbf{X}(\mathbb Z): \ \iota(a)=(a_{ij})\in M_n(\mathbb R), \ \sqrt{\sum\limits_{1\leq i, j\leq n}a_{ij}^{2}} \leq T\} $$
and 
$$ \textbf{X}(\mathbb{R},T)=\{ a\in\textbf{X}(\mathbb R): \ \iota(a)=(a_{ij})\in M_n(\mathbb R), \ \sqrt{\sum\limits_{1\leq i, j\leq n}a_{ij}^{2}} \leq T\}$$
for $T>0$. 
\end{Def}

The following result is an analogy of  \cite[Theorem 6.1]{ref11}.

\begin{theorem}
   Let $\mathcal{A}$ be  a central simple algebra of degree $n$ over $\mathbb{Q}$ satisfying $\mathcal A\otimes_{\mathbb Q} \Bbb R \cong M_n(\Bbb R)$ and $p(\lambda)$ be a monic irreducible polynomial of degree $n$ over $\Bbb Z$.  If $\bf X$ is an integral model of $X_{\mathcal A, p(\lambda)}$ defined in Definition Section 1 with ${\bf X}(\Bbb Z)\neq \emptyset$, then
 $$ N(\textbf{X},T)\thicksim (\prod\limits _{p \ \text{primes}}\int\limits _{\textbf{X}(\mathbb{Z}_{p})}d_{p})\cdot \int\limits_{\textbf{X}(\mathbb{R},T)}d_{\infty} $$  as $T\rightarrow \infty$,
 where $d_{p}$ and $d_{\infty}$ are the Tamagawa measures on $X_{\mathcal A, p(\lambda)}$.
 \end{theorem}
 \begin{proof} Let $K=\Bbb Q[\lambda]/(p(\lambda))$ and $\Upsilon=R_{K/\Bbb Q}^{1}(\Bbb G_m) $ be the norm one torus over $\Bbb Q$. 
 Since 
 $X_{\mathcal A, p(\lambda)} \cong SG_{\mathcal A}/\Upsilon$,  there is a canonical homomorphism 
 $$ \delta: \ H^1(\Bbb Q, \widehat{\Upsilon}) \cong \ker (\widehat{\Gal(\bar{\Bbb Q}/\Bbb Q)} \rightarrow \widehat{\Gal(\bar{\Bbb Q}/K)}) \longrightarrow \Br(X_{\mathcal A, p(\lambda)}) $$
which induces an isomorphism 
$$ \label{br} H^1(\Bbb Q, \widehat{\Upsilon}) \cong \ker (\widehat{\Gal(\bar{\Bbb Q}/\Bbb Q)} \rightarrow \widehat{\Gal(\bar{\Bbb Q}/K)}) \cong  \Br(X_{\mathcal A, p(\lambda)})/\Br(\Bbb Q) $$ by  \cite[Proposition 2.9 and Proposition 2.10]{ref5} and the proof of \cite[Theorem 6.1]{ref11}.  If we fix an identification of
  $\Bbb Q/\Bbb Z$ with all roots of unity in $\Bbb C^\times$, then all characters  \begin{equation} \label{cha} \chi \in \ker (\widehat{\Gal(\bar{\Bbb Q}/\Bbb Q)} \rightarrow \widehat{\Gal(\bar{\Bbb Q}/K)}) \end{equation} are locally constant functions over with $\textbf{X}(\Bbb Z_p)$ for all primes $p$ and $\textbf{X}(\Bbb R)$ by combining the map $\delta$ with the Brauer-Manin pairing. 
By \cite[Theorem 2.4]{ref9} and \cite[Theorem 4.3]{ref11}, one has
$$ N(\textbf{X},T)\thicksim  \sum_{\chi} (\prod_{p \ \text{primes}} N_p (\textbf{X}, \chi)) \cdot  N_{\Bbb R}(\textbf{X}, \chi, T) $$
as $T\rightarrow \infty$, where 
$$ N_p (\textbf{X}, \chi)= \int_{\textbf{X}(\Bbb Z_p)} \chi \cdot d_p \ \ \ \text{and} \ \ \ N_{\Bbb R}(\textbf{X}, \chi, T)=\int_{\textbf{X}(\mathbb{R},T)} \chi \cdot d_{\infty} $$ 
and $\chi$ runs over all characters in (\ref{cha}).  

 For any $\chi \neq 1$,  there is a finite non-trivial abelian extension $F/\Bbb Q$ such that $\ker(\chi)=\Gal(\bar{\Bbb Q}/F)$. By \cite[Chapter V, \S 4, Corollary 1]{Weil}, there is a prime $l$ ramified in $F/\Bbb Q$. This implies that there is $u_l\in\Bbb Z_l^\times$ such that $\chi(\rho(u_l))\neq 1$ where $\rho$ is 
 \[ \begin{CD} 
 \Bbb Q_l^\times @>\rho>> \Gal (\Bbb Q_l^{ab}/\Bbb Q_l) \\
@VVV @VVV \\
\Bbb I_{\Bbb Q}  @>>\rho> \Gal(\Bbb Q^{ab} /\Bbb Q) \end{CD} \]
 is the Artin map.  By \cite[Chapter V, Corollary1.2]{Neu}, there is $\alpha\in (\mathfrak o_{\mathcal A}\otimes_{\Bbb Z} \Bbb Z_l)^\times$ such that $N(\alpha)=u_l$. 
Since $\alpha$ acts on $ \textbf{X}(\Bbb Z_l)$, one obtains
$$ \int_{\textbf{X}(\Bbb Z_l)}\chi(y) d_l(y) = \int_{\textbf{X}(\Bbb Z_l)}\chi(\alpha\circ y) d_l(y) $$
 by using this action. 

By Galois cohomology with $$H^1(\Bbb Q_p, R_{K/\Bbb Q}(\Bbb G_m))=H^1(\Bbb Q_p, \Bbb G_m)=\{1\}, $$ one has 
the following commutative diagram
        
        \[
          \begin{CD}
                   @. 1 @. 1\\
                   @. @VVV @VVV \\
            1 @>>> \Upsilon(\Bbb Q_p) @>>> SG_{\mathcal A}(\Bbb Q_p) @>>>  X_{\mathcal A, p(\lambda)}(\Bbb Q_p)  @>{ev}>> H^1(\Bbb Q_p, \Upsilon)\\
            @. @V  V V @V V V @| @.\\
            1 @>>> (K\otimes_\Q \Bbb Q_p)^\times @>>> G_{\mathcal{A}}(\Bbb Q_p) @>>> X_{\mathcal A, p(\lambda)}(\Bbb Q_p) @>>> 1\\
            @. @V  V N_{K/\Q} V @V V N V @.\\
            1 @>>> \Bbb Q_p^\times @= \Bbb Q_p^\times \\
            @. @V V V @V V V @.\\
            @. H^1(\Bbb Q_p,\Upsilon) @. 1 \\
            @. @ VVV @. \\
             @. 1 
          \end{CD}
        \]
where $N$ is the reduced norm map. For any $y\in X_{\mathcal A, p(\lambda)}(\Bbb Q_p)$, there is $g\in G_{\mathcal A}(\Bbb Q_p)$ such that $y = g x g^{-1}$. Then 
\begin{equation} \label{ev} ev(y) = N(g) \cdot  N_{K/\Bbb Q} (K\otimes_\Q \Bbb Q_p)^\times \in H^1(\Bbb Q_p,\Upsilon)=\Bbb Q_p^\times/N_{K/\Bbb Q} (K\otimes_\Q \Bbb Q_p)^\times \end{equation}
 by tracing the above diagram.  Since the cup products and the Brauer-Manin pairing satisfy the following commutative diagram of pairings 
   \begin{equation}\label{pairing}
          \begin{CD}
           X_{\mathcal A, p(\lambda)}(\Bbb Q_p)  @.  \times  \ \ \ \Br(X_{\mathcal A, p(\lambda)})  @>>> \Br(\Bbb Q_p) \\
            @V{ev} V  V   @ AA \delta A  @ VV{id}V \\
       H^1(\Bbb Q_p,  \Upsilon) \ @. \times  \ H^1(\Bbb Q_p, \widehat{\Upsilon})      @>{\cup}>> \Br(\Bbb Q_p) \\
      @AAA   @ VVV @ VV{-id}V \\
       H^0 (\Bbb Q_p,  \Bbb G_m) \ @. \times \ \ H^2(\Bbb Q_p, \Bbb Z)      @>\cup>> \Br(\Bbb Q_p)
   \end{CD} 
        \end{equation}
by  \cite[Proposition 2.9]{ref5} and \cite[Chapter I, (1.4.7) Proposition]{ref6}, one obtains
  $$ \int_{\textbf{X}(\Bbb Z_l)}\chi(\alpha\circ y) d_l(y) = \chi(\rho(u_l))^{-1}  \int_{\textbf{X}(\Bbb Z_l)}\chi(y) d_l(y) $$ by the formula (\ref{ev}) and  \cite[Chapter VII, (7.2.12) Proposition]{ref6}. Therefore
   $$\int_{{\bf {X}}(\Bbb Z_l)} \chi \cdot d_l=0$$ and the result follows.
        \end{proof}

\section{Local computation}

In this section, we determine the orbits of $\text{Aut}(\mathcal{O}_{\mathcal{A}_{p}})$ under the conjugate action over $\textbf{X}(\mathbb{Z}_{p})$ where $\bf X$ is an integral model of $X_{\mathcal A, p(\lambda)}$ defined in Section 1 under the assumption of $\mathcal{A}_{p}$ is a matrix algebra or a division algebra for all primes $p$, the ring $\mathbb{Z}[\lambda]/(p(\lambda))$ is integrally closed and ${\bf X}(\Bbb Z)\neq \emptyset$.\\

When $\mathcal{A}_{p}$ is a matrix algebra, such a problem was already studied in \cite{ref3}.\\

\begin{lemma}
    Let $R$ be a PID and $f(x)$ be a monic polynomial over $R$ of degree $n$ with $(f(x),f'(x))=1$. If $GL_{n}(R)$ acts on the set of $n\times n$ matrices $M_{n}(R)$ over $R$ by conjugation, then there is one to one correspondence\begin{center}
        $\{A\in M_{n}(R):det(xI_{n}-A)=f(x)\}/GL_{n}(R)\leftrightarrow \{I\ ideals\ in\ R[x]/(f(x)):rank_{R}(I)=n\}/\sim $
    \end{center}
    where $I\sim J$ for two ideals means that there are $\alpha$ and $\beta$ which are not zero divisors such that $\alpha I=\beta J$.
\end{lemma}
\begin{proof}
    See\cite[Lemma6.4]{ref11}.
\end{proof}

When $\mathcal{A}_{p}$ is a division algebra, the number of the orbits is determined by the ramification as the following lemma:

\begin{lemma}
    Let $D_p$ be a central division algebra over $\Bbb Q_p$ of degree $n$ and $\Delta_p$ be the unique maximal order of $D_p$. Suppose that $p(\lambda)$ is a characteristic polynomial of an element in $\Delta_p$ over $\Bbb Z_p$ with $(p(\lambda), p'(\lambda))=1$ and 
$$ \mathbf X(\Bbb Z_p) =  \{ x\in \Delta_p: \ f_x(\lambda)=p(\lambda) \}   $$ with the action of $\Delta_p^\times$ by conjugation.
Then $p(\lambda)$ is irreducible over $\Bbb Z_p$ and the number of orbits by $\Delta_p^\times$ action inside $\mathbf X(\Bbb Z_p)$ is $e_p^{-1} n$ where $e_p$ is ramification index of extension $\Bbb Q_p[\lambda]/(p(\lambda))$ over $\Bbb Q_p$. 
\end{lemma}

\begin{proof} Let $x_0\in \Lambda$ such that $f_{x_0}(\lambda)=p(\lambda)$. Since $(p(\lambda), p'(\lambda))=1$, one concludes that $p(\lambda)$ is also the minimal polynomial of $x_0$. Since there is no zero divisors in $D_p$, one obtains that $p(\lambda)$ is irreducible over $\Bbb Z_p$.  For any $x\in  \mathbf X(\Bbb Z_p) $, there  is $g\in D_p$ such that $x=g \cdot x_0 \cdot g^{-1}$ by Skolem-Noether Theorem (see \cite[(7.21) Theorem] {Rein}). This implies that $D_p^\times$ acts on $\mathbf X(\Bbb Z_p)$ transitively by conjugation. 
    By \cite[(13.2) Theorem]{Rein}, one has  $$D_p^\times = \bigcup_{i=-\infty}^{\infty} \pi_p^i  \Delta_p^\times$$ where $\pi_p$ is a prime element of $\Delta_p$.  Since the stabilizer of $x_0$ by action of $D_p^\times$ is $$ \Bbb Q_p(x_0)^\times  \cong (\Bbb Q_p[\lambda]/(p(\lambda)))^\times $$ where $\Bbb Q_p[\lambda]/(p(\lambda))$ is a finite extension of degree $n$ over $\Bbb Q_p$ with the ramification index $e_p$, the number of orbits by $\Delta_p^\times$ action inside $\mathbf X(\Bbb Z_p)$ is equal to $$[D_p^\times: \Bbb Q_p(x_0)^\times \cdot \Delta_p^\times]= \frac{n}{e_p} $$ by \cite[(14.3) Theorem]{Rein} as desired.
    \end{proof}

For general central simple algebra $\mathcal A_p$ of degree $n$ over $\mathbb Q_p$. Then $\mathcal A_p \cong M_{m_p} (D_p)$ for some positive integer $m_p$ where $D_p$ is a central division algebra of degree $d_p$ over $\mathbb{Q}_p$. Let $\Delta_p$ be the unique maximal order of $D_p$. Then $ M_{m_p} (\Delta_p)$ is a maximal order of $\mathcal A_p$.\\

We set
$$\mathcal{S}_{x_{0}}:=\{g\in GL_{m_p} (D_p):g^{-1}\cdot x_{0}\cdot g\in M_{m_p} (\Delta_p)\}$$
where $x_{0}\in\textbf{X}(\mathbb{Z}_{p})$.\\

The set $\mathcal{S}_{x_{0}}$ is crucial for computing the number of the orbits of $\text{Aut}(\mathcal{O}_{\mathcal{A}_{p}})$ under the conjugate action over $\textbf{X}(\mathbb{Z}_{p})$. An important observation is that the set
$$D_{p}^{\times}\cdot GL_{m_p}(\Delta_{p})$$
is contained in $\mathcal{S}_{x_{0}}$, since 
$$xyx^{-1}\in GL_{m_p} (\Delta_p) \ \ \ \text{ for } \ \forall x\in D_p^\times \ \text{and} \   \forall y\in M_{m_p} (\Delta_p).$$
In particular, $D_{p}^{\times}\cdot GL_{m_p}(\Delta_{p})$ is in fact a subgroup of $GL_{m_p}(D_p)$.\\

We set
$$G:=D_{p}^{\times}\cdot GL_{m_p}(\Delta_{p}).$$

\begin{proposition}
    The number of the orbits of $G$ under the conjugate action over $\textbf{X}(\mathbb{Z}_{p})$ is equal to one if and only if 
    $$\mathcal{S}_{x_{0}}=\textrm{stab}_{G}(x_{0})\cdot G.$$
\end{proposition}

\begin{proof}
    If the number of the orbits of $G$ under the conjugate action over $\textbf{X}(\mathbb{Z}_{p})$ is one, for any $g\in\mathcal{S}_{x_0}$, set $$x=g^{-1}\cdot x_{0}\cdot g\in \textbf{X}(\mathbb{Z}_{p}).$$
    then $x=h^{-1}\cdot x_{0}\cdot h$ for some $h\in G$.\\
    
    One gets
    $$x_{0}=hg^{-1}\cdot x_{0}\cdot gh^{-1}\Rightarrow gh^{-1}\in\textrm{stab}_{G}(x_{0})$$
    since $(p(\lambda), p'(\lambda))=1$. That means $g\in\textrm{stab}_{G}(x_{0})\cdot G$. 
    
    Then one obtains
    $$\mathcal{S}_{x_{0}}=\textrm{stab}_{G}(x_{0})\cdot G,$$
    as desired by 
    $$\textrm{stab}_{G}(x_{0})\cdot G\subseteq \mathcal{S}_{x_{0}}$$ which is obvious since $G\subseteq \mathcal{S}_{x_{0}}$.\\

Conversely, if 
$$\mathcal{S}_{x_{0}}=\textrm{stab}_{G}(x_{0})\cdot G.$$

For any $x\in\textbf{X}(\mathbb{Z}_{p})$, there is $g\in GL_{m_{p}}(D_{p})$ such that $x=g^{-1}\cdot x_{0}\cdot g$. Then
$$x\in\textbf{X}(\mathbb{Z}_{p})\subseteq M_{m_{p}}(\Delta_{p})\Rightarrow g\in\mathcal{S}_{x_{0}}.$$

Write $g=\xi \cdot h$ for some $\xi\in\textrm{stab}_{G}(x_{0})$ and $h\in G$. One has
$$x=g^{-1}\cdot x_{0}\cdot g=h^{-1}\xi^{-1}\cdot x_{0}\cdot \xi h=h^{-1}\cdot x_{0}\cdot h$$
, i.e. the operation of $G$ on $\textbf{X}(\mathbb{Z}_{p})$ is transitive.

\end{proof}

The stabilizer of $x_0$ by action of $GL_{m_p} (D_p)$ can be identified with $(\mathbb{Q}_{p}\otimes \mathbb{Q}[\lambda]/(p(\lambda)))^{\times}$.\\

When $\mathcal{A}$ is a matrix algebra or a division algebra, one gets
\begin{equation}
    \mathcal{S}_{x_{0}}=(\mathbb{Q}_{p}\otimes \mathbb{Q}[\lambda]/(p(\lambda)))^{\times}\cdot (D_{p}^{\times}\cdot GL_{m_p}(\Delta_{p}))
\end{equation}by Lemma 3.1, Lemma 3.2 and Proposition3.3.

\begin{remark}
    It is natural to expect that (3.1) holds for arbitrary central simple algebras. However, the following example says that it fails even for $m_{p}=2$.
\end{remark}

 \begin{example}
    Let $D_{p}$ be a quaternion algebra over $\mathbb{Q}_{p}$, $\mathcal{A}_{p}=M_{2}(D_{p})$ and $p(\lambda)=\lambda^{4}-ap^{2}$, where $p$ is an odd prime and $a$ is a square-free integer such that
    $$(\frac{a}{p})=-1.$$
    Then $p(\lambda)$ splits into two Eisenstein polynomials over an unramified quadratic splitting field of $\mathcal{A}_{p}$ i.e. $L=\mathbb{Q}_{p}(\sqrt{a})$. $D_{p}$ can be identified with a subring of $M_{2}(L)$, as:
    $$D_{p}=\{\begin{pmatrix}
        u & p\cdot\sigma(v) \\ v& \sigma(u)
    \end{pmatrix}:u,v\in L\}$$
    where $\sigma$ is the Frobenius in $\Gal(L/\mathbb{Q}_{p})$.
    Set
    $$x_{0}=\begin{pmatrix}
        0 & \sqrt{a} \\ p& 0
    \end{pmatrix}\in M_{2}(D_{p})=\mathcal{A}_{p},g=\begin{pmatrix}
        1 &  \\ & \pi
    \end{pmatrix}\in GL_{2}(D_{p})$$
    where $$\pi=\begin{pmatrix}
        0 & p \\ 1& 0
    \end{pmatrix}\ is\ a\ prime\ of\ D_{p}.$$
One has $g\in S_{x_{0}}$, since
$$g^{-1}\cdot x_{0}\cdot g=\begin{pmatrix}
    1 &  \\ & \pi^{-1}
\end{pmatrix}\cdot \begin{pmatrix}
    0 & \sqrt{a} \\ p& 0
\end{pmatrix}\cdot\begin{pmatrix}
    1 &  \\ & \pi
\end{pmatrix}=\begin{pmatrix}
    0 & \sqrt{a}\cdot\pi \\ \pi& 0
\end{pmatrix}\in M_{2}(\Delta_{p}).$$
While $$g\notin (\mathbb{Q}_{p}[\lambda]/(p(\lambda)))^{\times}\cdot (D_{p}^{\times}\cdot GL_{2}(\Delta_{p})),$$
since $v_{p}(\det (g))=1$ and the valuation of the determinant of the elements which belong to the right set is always even.

 \end{example}

We also give a special case for determining the set
$$ \mathcal{S}_{\Delta_{p}}:=\{ g\in GL_{m_p}(D_p) :  g M_{m_p} (\Delta_p) g^{-1}  = M_{m_p}(\Delta_p)\}.$$
 \begin{example}
    For general central simple algebra $\mathcal A_p = M_{m_p} (D_p)$ of degree $n$, where $D_{p}$ is a division algebra over $\mathbb{Q}_{p}$ of degree $r_{p}$. We can identify $\mathcal{A}_{p}$ with a subring of $M_{n}(L)$, where $L$ is the unique unramified extension over $\mathbb{Q}_{p}$ of degree $n$. Let $p(\lambda)=\lambda^{n}-p$ and
    $$x_{0}=\left(\begin{array} [c]{llll}
        0 \ \  0 \ \cdots \ \ 0 \ \ \pi  \\ 1 \ \ 0 \ \cdots  \ \ 0 \ \ 0 \\
         \ \ \ \cdots \ \cdots \ \cdots \\
        0 \ \ 0 \ \cdots \ \ 1 \ \ 0
        \end{array}
        \right)_{m_{p}\times m_{p}}\in \mathcal{A}_{p}$$
        where $$\pi=\left(\begin{array} [c]{llll}
            0 \ \  0 \ \cdots \ \ 0 \ \ p  \\ 1 \ \ 0 \ \cdots  \ \ 0 \ \ 0 \\
             \ \ \ \cdots \ \cdots \ \cdots \\
            0 \ \ 0 \ \cdots \ \ 1 \ \ 0
            \end{array}
            \right)_{r_{p}\times r_{p}}\in D_{p}.$$
We set
$$\mathcal{S}_{x_{0}}^{\prime}:=\{g\in GL_{n}(L):g^{-1}\cdot x_{0}\cdot g\in M_{n}(\mathcal{O}_{L})\}$$
where $\mathcal{O}_{L}$ is the valuation ring of $L$.\\

By Lemma3.1, the operation of $GL_{n}(\mathcal{O}_{L})$ on $\textbf{X}(\mathcal{O}_{L})$ is transitive. Then one can get a result which is an analogy of Proposition3.3
$$\mathcal{S}_{x_{0}}^{\prime}=(L[\lambda]/(p(\lambda)))^{\times}\cdot GL_{n}(\mathcal{O}_{L})$$

For any $$g\in\mathcal{S}_{x_{0}}\subseteq \mathcal{S}_{x_{0}}^{\prime}$$
Write $g=\xi_{1}\cdot h_{1}$ for some $\xi_{1}\in(L[\lambda]/(p(\lambda)))^{\times}$ and $h_{1}\in GL_{n}(\mathcal{O}_{L})$.\\

One has $$(L[\lambda]/(p(\lambda)))^{\times}=L[x_{0}]^{\times}\ \ \ and\ \ \ \xi_{1}=x_{0}^{l}\cdot u_{1}$$
for some $u_{1}\in\mathcal{O}_{L}[x_{0}]^{\times}\subseteq GL_{n}(\mathcal{O}_{L})$ and integer $l$, since $p(\lambda)$ is an Eisenstein polynomial over $L$ and $x_{0}$ is a prime of $L[x_{0}]$. Then
$$g=x_{0}^{l}\cdot (u_{1}\cdot h_{1})$$
as both $g$ and $x_{0}$ are contained in $GL_{m_{p}}(D_{p})$, one has
$$u_{1}\cdot h_{1}\in GL_{m_{p}}(D_{p})\cap GL_{n}(\mathcal{O}_{L})=GL_{m_{p}}(\Delta_{p})$$
i.e. $$g\in (\mathbb{Q}_{p}[\lambda]/(p(\lambda)))^{\times}\cdot (D_{p}^{\times}\cdot GL_{m_p}(\Delta_{p}))=\mathbb{Q}_{p}[x_{0}]^{\times}\cdot (D_{p}^{\times}\cdot GL_{m_p}(\Delta_{p})).$$
That means $$\mathcal{S}_{\Delta_{p}}\subseteq \mathbb{Q}_{p}[x_{0}]^{\times}\cdot (D_{p}^{\times}\cdot GL_{m_p}(\Delta_{p})).$$

Set $$y_{0}=\left(\begin{array} [c]{llll}
    0 \ \  0 \ \cdots \ \ 0 \ \ 1  \\ 1 \ \ 0 \ \cdots  \ \ 0 \ \ 0 \\
     \ \ \ \cdots \ \cdots \ \cdots \\
    0 \ \ 0 \ \cdots \ \ 1 \ \ 0
    \end{array}
    \right)_{m_{p}\times m_{p}}\in \mathcal{A}_{p}.$$
Write
    $$(a_{ij}^{(k)})_{m_{p}\times m_{p}}=x_{0}^{-k}\cdot y_{0}\cdot x_{0}^{k}$$
where $k=1,2,\cdots,m_{p}-1$. One has $$a_{m_{p}-k+1,m_{p}-k}^{(k)}=\pi^{-1}\notin \Delta_{p}\Rightarrow x_{0}^{k}\notin\mathcal{S}_{\Delta_{p}},$$
and for any $\xi_{0}\in\mathbb{Q}_{p}[x_{0}]^{\times}$,
$$\xi_{0}=u_{0}\cdot x_{0}^{k}\ \ \ for\ \ \ u_{0}\in\mathbb{Z}_{p}^{\times},$$
and $\mathbb{Z}_{p}^{\times}$ is contained in the center of $\GL_{m_{p}}(D_{p})$.\\
Then we get the result
$$\mathcal{S}_{\Delta_{p}}=D_{p}^{\times}\cdot GL_{m_p}(\Delta_{p}).$$

 \end{example}

\section{The asymptotic formula}

In this section, we compute the asymptotic formula for $N_{\mathfrak o_{\mathcal A} }(p(\lambda,T))$ under the assumption of $\mathcal{A}_{p}$ is a matrix algebra or a division algebra for all primes $p$, $\mathcal{A}$ is split by $\mathbb{R}$, the ring $\mathbb{Z}[\lambda]/(p(\lambda))$ is integrally closed and ${\bf X}(\Bbb Z)\neq \emptyset$. To do this we need to compute the integrals in each term of Theorem 2.1.\\

\begin{theorem}
    Let $\mathcal{A}$ be  a central simple algebra of degree $n$ over $\mathbb{Q}$ satisfying $\mathcal A\otimes_{\mathbb Q} \Bbb R \cong M_n(\Bbb R)$ and $\mathcal{A}\otimes \mathbb{Q}_{p}$ is a matrix algebra or a division algebra over $\mathbb{Q}_{p}$ for all primes $p$, if the ring $\mathbb{Z}[\lambda]/(p(\lambda))$ is integrally closed and ${\bf X}(\Bbb Z)\neq \emptyset$, then
    $$N_{\mathfrak o_{\mathcal A} }(p(\lambda), T)\thicksim (\prod\limits_{p\in S}\frac{n}{e_{p}\cdot \prod\limits_{i=1}^{n-1}(1-\frac{1}{p^{i}})})\cdot\frac{2^{r_{1}}(2\pi)^{r_{2}}h_KR_K\omega_n}{w|d_{K}|^{\frac{1}{2}}\prod\limits_{i=2}^{n}\Lambda (\frac{i}{2})} \cdot T^{\frac{n(n-1)}{2}}$$
    as $T\rightarrow \infty$, where $S=\{p:\mathcal{A}\otimes \mathbb{Q}_{p}\ is\ a\ division\ algebra\ over\ \mathbb{Q}_{p}\}$, $e_p$ is ramification index of extension $\Bbb Q_p[\lambda]/(p(\lambda))$ over $\Bbb Q_p$, $h_K$, $R_K$, $r_1$, $r_2$, $w$ and $d_{K}$ are the class number, the regulator, the number of real places, the number of complex places, the number of roots of unity in $K$, the discriminant of number field $K=\mathbb{Q}[X]/(p(X))$ respectively, $\omega_n$ is the volume of the unit ball in $\mathbb{R}^{\frac{n(n-1)}{2}}$ and $\Lambda (s)=\pi^{-s}\varGamma (s) \zeta(2s)$.
\end{theorem}

\begin{proof}

   When $\mathcal{A}_{p}=\mathcal{A}\otimes \mathbb{Q}_{p}$ is a matrix algebra over $\mathbb{Q}_{p}$. By Lemma3.1, one obtains

    $$1\rightarrow  \textrm{stab}_{\text{Aut}(\mathcal{O}_{\mathcal{A}_{p}})}(x_{0})\rightarrow \text{Aut}(\mathcal{O}_{\mathcal{A}_{p}}) \rightarrow \textbf{X}(\mathbb{Z}_{p})\rightarrow 1$$
where $x_{0}\in\textbf{X}(\mathbb{Z}_{p})$.\\

For $(p(\lambda), p'(\lambda))=1$, one concludes the stabilizer of $x_0$ by action of $\text{Aut}(\mathcal{A}_{p})=GL_{n}(\mathbb{Q}_{p})$ is
$$\Bbb Q_p(x_0)^\times  \cong (\Bbb Q_p[\lambda]/(p(\lambda)))^\times.$$
Since the ring $\mathbb{Z}[\lambda]/(p(\lambda))$ is integrally closed, one gets the stabilizer of $x_0$ by action of $\text{Aut}(\mathcal{O}_{\mathcal{A}_{p}})=GL_{n}(\mathbb{Z}_{p})$ is
$$\Bbb Z_p[x_0]^\times  \cong (\Bbb Z_p[\lambda]/(p(\lambda)))^\times\cong \prod\limits_{\mathfrak{p}|p}\mathcal{O}_{\mathfrak{p}}^{\times}$$
where $\mathfrak{p}$ is the primes in the field $K=\mathbb{Q}[\lambda]/(p(\lambda))$ above $p$ and $\mathcal{O}_{\mathfrak{p}}$ is the valuation ring of the completion $K_{\mathfrak{p}}$.\\

One has
$$\int\limits _{\textbf{X}(\mathbb{Z}_{p})}d_{p}=\frac{\mu_{p}(GL_{n}(\mathbb{Z}_{p}))}{\mu_{p}(\prod\limits_{\mathfrak{p}|p}\mathcal{O}_{\mathfrak{p}}^{\times})}=\frac{(1-\frac{1}{p})}{\prod\limits_{\mathfrak{p}|p}(1-\frac{1}{N_{\mathfrak{p}}})}\mu_{p}(SL_{n}(\mathbb{Z}_{p}))$$
where $\mu_{p}$ is the Haar measure and $N_{\mathfrak{p}}$ is the number of elements in the residue field of $\mathfrak{p}$. \\

Similarly, when $\mathcal{A}_{p}$ is a division algebra over $\mathbb{Q}_{p}$, i.e. $p\in S$. By Lemma3.2, one obtains
$$\int\limits _{\textbf{X}(\mathbb{Z}_{p})}d_{p}=\frac{n}{e_{p}}\cdot\frac{\mu_{p}(\Delta_{p}^{\times})}{\mu_{p}(\prod\limits_{\mathfrak{p}|p}\mathcal{O}_{\mathfrak{p}}^{\times})}$$
where $\Delta_{p}$ is the unique maximal order of $\mathcal{A}_{p}=D_{p}$ and $e_p$ is ramification index of extension $\Bbb Q_p[\lambda]/(p(\lambda))$ over $\Bbb Q_p$.\\

$\mathcal{A}_{p}=D_{p}\cong (\chi,p)$, where $(\chi,p)$ is the cyclic algebra defined by an unramified character $\chi$ and the prime $p$ of $\mathbb{Z}_{p}$. One can get a field $L$ by the character $\chi$ with integral bases :$l_{1},l_{2},\cdots,l_{n}$ over $\mathbb{Q}_{p}$. Set $e$ be the generator of $D_{p}$. If
    $$N_{\mathcal{A}_{p}}(\sum\limits_{i,j}X_{i,j}\cdot l_{i}\cdot e^{j})=\prod\limits_{\sigma\in \Gal(L/\mathbb{Q}_{p})}(\sum\limits_{i,j}X_{i,j}\cdot \sigma(l_{i})\cdot \sigma(e^{j}))\in\mathbb{Z}_{p}^{\times}$$
then $$\prod\limits_{\sigma\in \Gal(L/\mathbb{Q}_{p})}(\sum\limits_{i}(X_{i,0}\cdot \sigma(l_{i})))\in\mathbb{Z}_{p}^{\times}$$
since $N_{\mathcal{A}_{p}}(\sigma(e))\in p\mathbb{Z}_{p}$. That means 
$$N_{\mathcal{A}_{p}}(\sum\limits_{i}(X_{i,0}\cdot l_{i}))\in \mathbb{Z}_{p}^{\times}.$$
So $$\mu _{p}(\Aut(\mathcal{O}_{\mathcal{A}_{p}}))=\mu_{p}(\Delta_{p}^{\times})=1-\frac{1}{N_{\mathfrak{p}_{L}}}=1-\frac{1}{p^{n}}$$

By 14.9 in \cite{ref10}, one has

$$\mu _{p}(SL_{n}(\mathbb{Z}_p))=\prod\limits_{i=2}^{n}(1-\frac{1}{p^{i}})\ \ \ \text{and} \ \ \ \prod\limits_{p}\mu _{p}(SL_{n}(\mathbb{Z}_p))=\prod\limits_{i=2}^{n}\zeta(i)^{-1}.$$

Therefore,
$$\prod \limits_{p}\int\limits _{\textbf{X}(\mathbb{Z}_{p})}d_{p}=(\prod\limits_{p\in S}\frac{n}{e_{p}}\cdot\frac{1}{\prod\limits_{i=1}^{n-1}(1-\frac{1}{p^{i}})})\cdot\frac{\zeta _{K}(s)}{\zeta (s)}|_{s=1}\cdot\prod\limits_{i=2}^{n}\zeta(i)^{-1}$$
where $\zeta_{K}(s)$ is the Dedekind zeta function of $K$.\\

By the class number formula (see Chapter VII, (5.11) Corollary in \cite{Neu}), one has
$$\frac{\zeta _{K}(s)}{\zeta (s)}|_{s=1}=Res_{s=1}\zeta_{K}(s)=\frac{2^{r_1}(2\pi)^{r_2}hR}{w(d_{K})^{\frac{1}{2}}}.$$

Over $\mathbb{R}$, one has
$$1\rightarrow (\mathbb{R}^{\times})^{r_{1}}\times(\mathbb{C}^{\times})^{r_{2}}\rightarrow GL_{n}(\mathbb{R})\stackrel{\pi}{\rightarrow} \textbf{X}(\mathbb{R})\rightarrow 1$$
where $r_{1}$ and $r_{2}$ are the number of real places and the number of complex places respectively.\\

By the discussion of 3.2 in \cite{ref9}, one has the decomposition of $GL_{n}(\mathbb{R})$
$$1\rightarrow \Sigma\times K\rightarrow O(n)\times[(\mathbb{R}^{\times})^{r_{1}}\times(\mathbb{R}^{\times})^{r_{2}}\times (SO(2))^{r_{2}}\times N]\rightarrow GL_{n}(\mathbb{R})\rightarrow 1$$
where $$\Sigma=\{\text{diag}(\epsilon _{1},\cdots,\epsilon _{r_{1}},I_{2},\cdots,I_{2}):\epsilon_{i} =\pm 1\}$$
$$K=\{\text{diag}(1,\cdots,1,k_{1},\cdots,k_{r_{2}}):k_{i}\in SO(2)\}$$
$N$ is the upper triangular unipotent subgroup of $GL_{n}(\mathbb{R})$, $O(n)$, $SO(2)$ and $I_{2}$ are the orthogonal group in $GL_{n}(\mathbb{R})$, 
the special orthogonal group in $GL_{2}(\mathbb{R})$ and the identity matrix in $M_{2}(\mathbb{R})$ respectively.\\

Therefore, by Proposition A.3 in \cite{ref9}
$$\lim_{T\to \infty}\int\limits_{\textbf{X}(\mathbb{R},T)}d_{\infty}=\frac{1}{2^{r_{1}}}\cdot(\frac{\pi}{2}\cdot\frac{1}{2\pi})^{r_{2}}\lim_{T\to \infty} vol(\pi^{-1}(\textbf{X}(\mathbb{R},T)\cap O(n)))\cdot \lim_{T\to \infty} vol(\pi^{-1}(\textbf{X}(\mathbb{R},T)\cap N))$$
$$=\frac{1}{2^{n}}\lim_{T\to \infty} vol(\pi^{-1}(\textbf{X}(\mathbb{R},T)\cap O(n)))\cdot \lim_{T\to \infty} vol(\pi^{-1}(\textbf{X}(\mathbb{R},T)\cap N))$$
Since $O(n)$ is compact, one has
$$\lim_{T\to \infty} vol(\pi^{-1}(\textbf{X}(\mathbb{R},T)\cap O(n)))=vol(O(n))=2^{n}\pi^{\frac{n(n+1)}{4}}\prod\limits_{i=1}^{n}\Gamma(\frac{i}{2})^{-1}$$ by 14.12 in \cite{ref10}.\\

By Theorem2.1 and
$$vol(\pi^{-1}(\textbf{X}(\mathbb{R},T)\cap N))\thicksim\omega_{n}T^{\frac{n(n-1)}{2}}$$
as $T\rightarrow\infty$, one obtains the result.

\end{proof}

 \end{document}